\documentclass[11pt] {article}      %
\usepackage{amsmath,amssymb, amsthm, eucal,bm}  
\usepackage{graphicx, amscd}
\usepackage{framed}

\let\OLDthebibliography\thebibliography
\renewcommand\thebibliography[1]{
  \OLDthebibliography{#1}
  \setlength{\parskip}{3pt}
  \setlength{\itemsep}{0pt plus 0.3ex}
}

\usepackage{txfonts}                  
\usepackage[T1]{fontenc}     

\usepackage{scrextend}     

\usepackage{float}                 
\usepackage{tikz}
\usetikzlibrary{matrix, arrows, calc}

\theoremstyle{plain}
\newtheorem{theorem}{Theorem}[section]            
\newtheorem{proposition}[theorem]{Proposition}  

\theoremstyle{definition}

\numberwithin{theorem}{section}
\numberwithin{equation}{section}
\newcommand{\gaction}[2]{\genfrac{}{}{0.5pt}{}{#1}{#2}%
                        \!\lower2pt\hbox{\rotatebox[origin=c]{-90}{{$\looparrowright$}}}}
\newcommand{\dotfraction}[2]{\genfrac{}{}{0.5pt}{}{#1}{#2}%
                        \!\lower.5pt\hbox{{$\circ$}}}

\usepackage{titlesec}                                                           
\titlelabel{\thetitle.\ }                                                         
\titleformat*{\section}{\fontsize{14pt}{14pt} \bf}        
\usepackage{fancyhdr}
\usepackage{sidecap}  
\usepackage[hang,small,sc]{caption}

\def\SO{\hbox{SO}}
\def\SL{\hbox{SL}}
\makeatletter
\newcommand*\bigcdot{\mathpalette\bigcdot@{.41}}
\newcommand*\bigcdot@[2]{\mathbin{\vcenter{\hbox{\scalebox{#2}{$\m@th#1\bullet$}}}}}
\makeatother
\makeatletter
\newcommand*\bbigcdot{\mathpalette\bigcdot@{.61}}
\newcommand*\bbigcdot@[2]{\mathbin{\vcenter{\hbox{\scalebox{#2}{$\m@th#1\bullet$}}}}}
\makeatother
\begin{document}

\title{Spinors and Descartes configurations of disks}
\author{Jerzy Kocik 
    \\ \small Department of Mathematics, Southern Illinois University, Carbondale, IL62901
   \\ \small jkocik{@}siu.edu
}

\date{}

\maketitle

\begin{abstract}
We define spinors for pairs of tangent disks in the Euclidean plane
and prove a number of theorems, one of which may be interpreted as a 
``square root of Descartes Theorem''.
In any Apollonian disk packing, spinors form a network. 
In the Apollonian Window, a special case of Apollonian disk packing,
all spinors are integral. 
\\[5pt]
{\small {\bf Keywords:} Descartes configuration, Descartes formula, Apollonian disk packing, 
spinor, Pytha\-go\-rean triples,
Euclid's parametrization, Minkowski space.}
\\[5pt]
{\small {\bf MSC:} \ 52C26, 51N20, 11E88, 15A66,  51M15}
\end{abstract}

~\\\\

\section{Introduction} 

A {\bf tangency spinor} of an ordered pair of mutually tangent disks 
in the Euclidean plane (identified for convenience with complex plane, $\mathbb R^2 \cong \mathbb C$)  
is a vector (complex number)
\begin{equation}
\label{eq:start}
u \ = \ \pm \sqrt{\frac{z}{r_1r_2}}
\end{equation}
where $z=\overrightarrow{O_1O_2}$ is the vector (complex number) 
joining the centers $O_1$ and $O_2$ of the disks
of radii $r_1$ and $r_2$, respectively. 
The tangency spinor is defined up to a sign.
The concept was introduced  in \cite{jk-c}. 
\\

The definition might look at first somewhat unnatural or arbitrary.
Yet it leads to a number of amazing and fruitful properties, proved in the following sections.  
A quartet of  mutually tangent disks is called a Descartes configuration. 
The present paper proves that spinors  in a Descartes configuration 
admit a choice of signs 
such that these two properties hold
\begin{equation}
\label{eq:laws}
\begin{array}{cc}
    ``\hbox{curl}\; \mathbf u = 0'' :   &\qquad   u_{12} + u_{23} + u_{31} = 0\\[5pt]
    ``\hbox{div}\; \mathbf u = 0'' :   &\qquad   u_{14} + u_{24} + u_{34} = 0\\
\end{array}
\end{equation}
where $u_{ij}$ represents a spinor for  $i$-th and $j$-th tangent disks 
as shown in Figure \ref{fig:divcurl}.

\begin{figure}[h]
\centering
\includegraphics{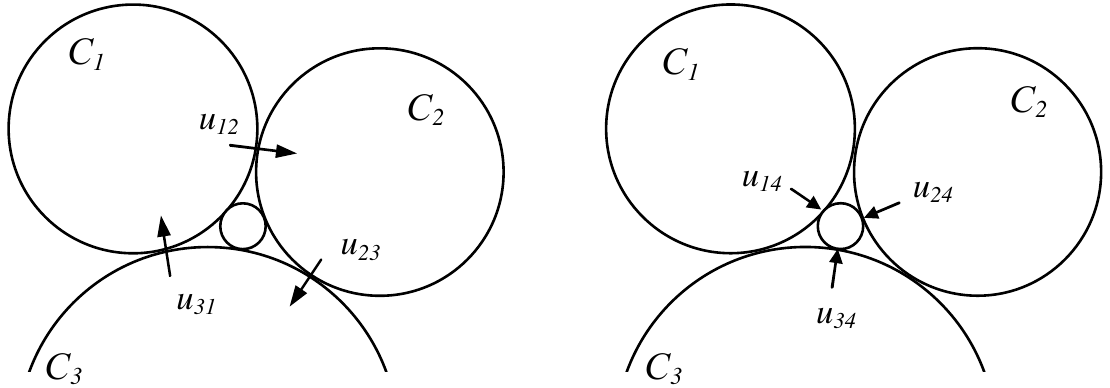}
\caption{Spinor theorems for a Descartes configuration}
\label{fig:divcurl}
\end{figure}
%
%
%
The arrow representation of the spinors in the figures are symbolic and mark only the order of the disks.  
The latter result, $``\hbox{div}\; \mathbf u = 0''$, may be viewed as a ``spinorial'' version of 
Descartes' theorem on circles, which is recalled below.
\\

The  curvatures of four mutually tangent circles (a so-called {\bf Descartes configuration}) 
satisfy Descartes' formula (1643)  \cite{Coxeter,Ped67,Sod}:
\begin{equation}
\label{eq:Descartes1}
                  (A+B+C+D)^2 = 2\; (A^2 + B^2 + C^2 + D^2 )
\end{equation}
The formula was the answer to {\bf Descartes problem}:
Given three mutually tangent circles of curvatures $A$, $B$, $C$,
find the fourth which completes them to the Descartes configuration. 
Due to the quadratic nature of (\ref{eq:Descartes1}), 
there are  two solutions:
\begin{equation}
\label{eq:Descartes1sol}
                  D \ = \ A+B+C \ \pm \  2\; \sqrt{(AB+ BC+ CA}
\end{equation}
As noticed by Boyd, a more convenient version of Descartes formula \cite{Boyd}
is a simple linear equation that uses both solutions, $D$ and $D'$: 
\begin{equation}
\label{eq:Boyd1}
                 D+D' = 2(A+B+C)
\end{equation}

An Apollonian disk packing (or simply ``gasket'')  is a fractal completion of a Descartes configuration. 
Such a packing is called {\bf integral} if the curvatures of all disks are integers.
The curvatures of all disks in the gasket are determined from any four disks by iterative use of (\ref{eq:Boyd1}).
Therefore integrality of the gasket follows from the integrality of the first four disks.
\\

In the Apollonian disk packing, spinors are defined at every point of tangency of two disks.
Figure \ref{fig:AWspinors} shows a the upper half of {\bf Apollonian Window}
 -- an exceptionally regular integral disk packing.
The big numbers represent the disk curvatures.
The arrows shows spinors (for clarity the brackets are omitted). 
Note that all spinors in the Apollonian Window $\mathcal A$ are  integral.


\begin{figure}[h]
\centering
\includegraphics[scale=.82]{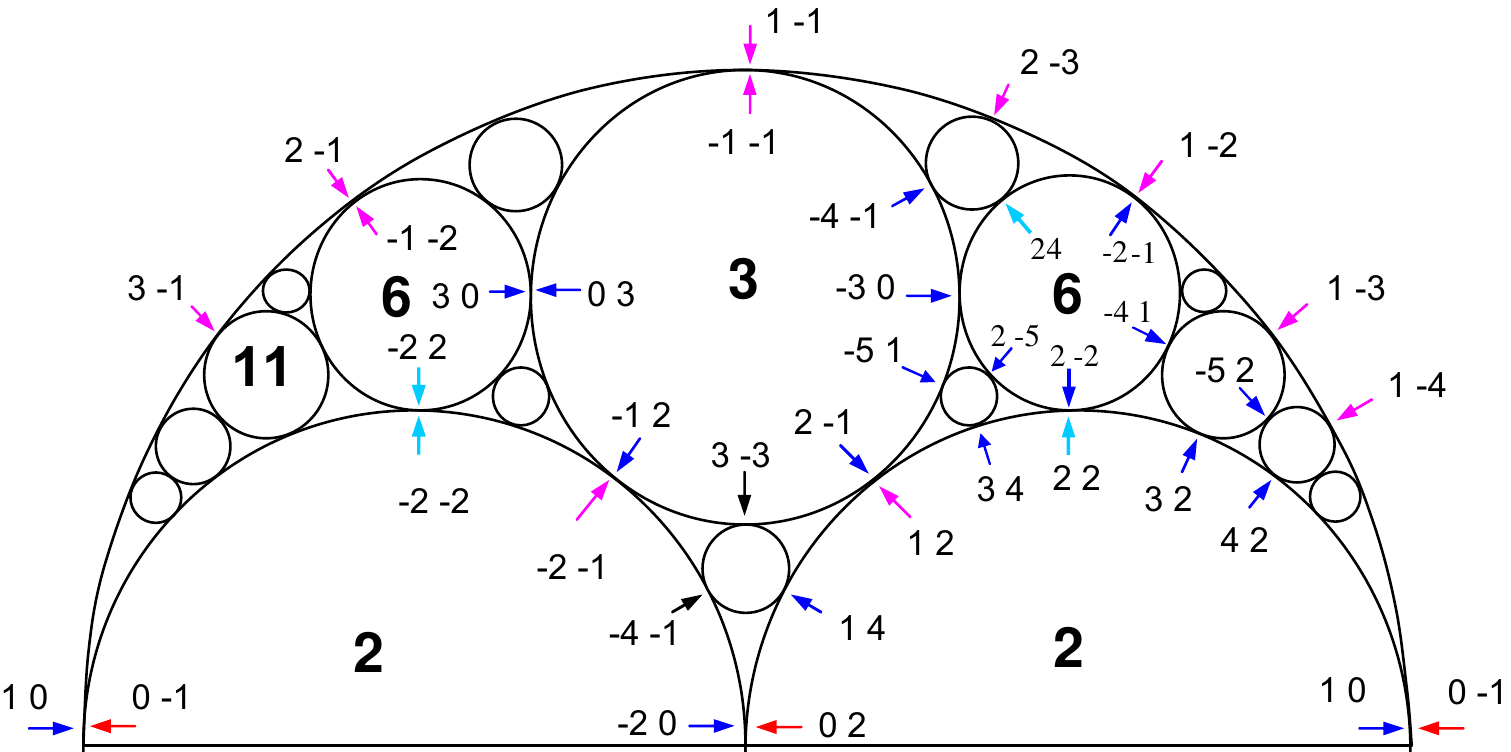}
\caption{\small Integral spinors in the Apollonian Window}
\label{fig:AWspinors}
\end{figure}

The initial motivation for tangency spinors was the discovery that the Apollonian Window 
(and other integral Apollonian disk packings) 
contain Pythagorean triples \cite {jk-c}.
Spinors were  introduced as the geometric representation of their Euclidean parametrizations.
The name ``spinor `' is justified by the fact that the latter may be viewed 
as the spinors for the Minkowski space $\mathbb R^{2+1}$,
quite analogously to spinors of the Minkowski space $\mathbb R^{3+1}$ modeling the physical space-time.
The clarification of these facts is postponed to last section 
for  a better flow of the exposition.

~

The laws \eqref{eq:laws} have local character.  
Extending them to the whole Apollonian gasket 
will be addressed in subsequent papers, where a spinor fiber bundle over 
an Apollonian disk packing will be discussed, 
including topological obstructions to a global extension of the features observed.


\newpage

\section{Spinor theorems for Descartes configurations}

In this section we present 
a sequence of theorems that are for the tangency spinors what Descartes Theorem is for circles.

\subsection{\bf \large Spinor space}

By a {\bf tangency spinor space} we understand the triple 
$\mathbb T = (\mathbb R^2, g, \omega)$, that is a two dimensional Euclidean space with a {\bf scalar product} $g$ 
and a {\bf symplectic} product $\omega$, which for two vectors
$$
u_1 \ = \ \begin{bmatrix}a_1\\b_1\end{bmatrix}
\qquad
u_2 \ = \ \begin{bmatrix}a_2\\b_2\end{bmatrix}
$$
are defined as follows:
\begin{equation}
\label{eq:products}
\begin{array}{ccl}
u_1 \/\bigcdot\/\, u_2  \ &=& \ a_1b_1 + a_2b_2  \\
u_i \times u_2 \ &=& \ a_1b_2 - a_2b_1  = \det \left[\begin{matrix}  a_1 &a_2\\
                                                                      b_1 & b_2  \end{matrix}\right]
\end{array}
\end{equation}
We use an alternative notation for the products:
$$
u_1 \bigcdot\,  u_2  \ \equiv g(u_1,\, u_2 ) 
\qquad \hbox{and}\qquad 
u_1 \times u_2  \ \equiv \omega(u_1,\, u_2 )  
$$
Clearly, $g$  defines a norm while $\omega$ vanishes for a repeated entry:  
$$
   \|\,u\,\|^2 \ =\ u\bigcdot u \,, \qquad\qquad u\times u \ = \ 0 \,.
$$
Moreover, we define a {\bf spinor conjugation} 
\begin{equation}
\label{eq:i}
u \mapsto u' = \left[\begin{matrix}  0 &-1\\
                                                     1 &\phantom{-}0  \end{matrix}\right]
                                                     \left[\begin{matrix}  a\\  b  \end{matrix}\right] = 
                                                    \left[\begin{matrix}  -b\\ \phantom{-}a  \end{matrix}\right] 
\end{equation}
which is essentially a complex strucure defined by the products by
\begin{equation}
\label{eq:gomega1}
u_1\times u_2 =  u'_1 \bigcdot\, u_2  \,.
\end{equation}

~

Alternatively, we may view the spinor space 
as a one-dimensional Hilbert space $\mathbb T \cong \mathbb C$.   \ 
For $u_1 = a_1 + b_1i$  and   $u_2 = a_2 + b_2i$, 
one defines the {\bf Hermitian product} as
$$
                           \bar u_1 u_2    \   = \   \underbrace{(a_1 b_1  +  a_2 b_2)}_{g(u_1,\,u_2)}   
                                                      \  + \   \underbrace {( a_1b_2  -  a_2 b_1)}_{\omega(u_1,\,u_2)} \; i
$$
The above spinor products  of two complex numbers  
are defined in this context as real numbers
$$
\begin{array}{rll}
 		g (u_1,\, u_2)  \          &= \ \ \frac{1}{2} (  \bar u_1u_2  + u_1 \bar  u_2 ) \ &= \ a_1b_1 + a_2b_2  \,, \\[5pt]
 		\omega (u_1,\, u_2)  \ &= \   \frac{1}{2i}\; (\bar u_1 u_2 - u_1 \bar u_2  )    \ &= \ a_1b_2 - a_2b_1   
\end{array} 		
$$
The norm is now $\|\,u\,\|^2 = |u|^2$.
Spinor conjugation \eqref{eq:i} becomes $u \mapsto u' = iu$, and 
relation \eqref{eq:gomega1} becomes $\omega (u_1, u_2) =  g (i u_1, u_2)$.
\\

We are ready to investigate the behavior of spinors in the systems of tangent disks. 
The two formulations, real and complex,  will be used interchangeably.

\subsection{\bf \large Seven theorems}

{\bf Notation:}  In the rest of this paper the name of a circle and its curvature 
will be denoted by the same capital letter.  
Spinors will be denoted by the small letters.
\\

\begin{figure}[h]
\centering
\includegraphics[scale=.75]{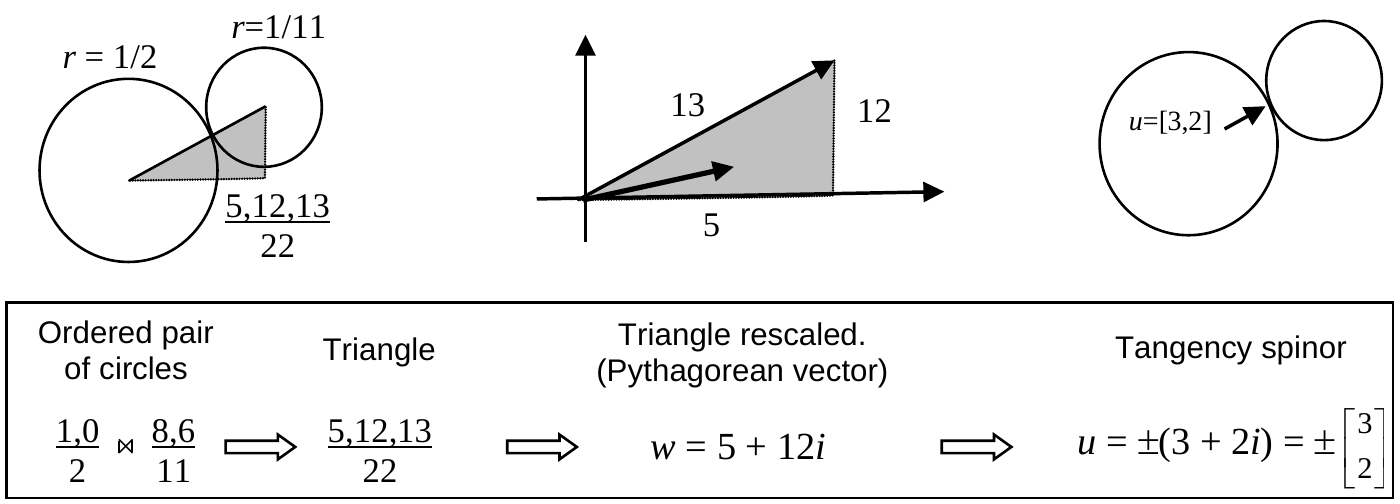}
\caption{\small From two tangent circles to a spinor (not to scale)}
\label{fig:process}
\end{figure}

Recall from the introduction that if $A$ and $B$ are an ordered pair of tangent circles with 
centers $w_1, w_2\in \mathbb C$ and curvatures $A, B\in \mathbb R$, 
respectively, then we define the {\bf spinor of tangency} as a complex number 
$u\in \mathbb C$ such that
\begin{equation}
\label{eq:spinor}
       u^2  =z\, A \, B\,,
\end{equation}
where $z=w_2-w_1$.  Spinors are defined up to a sign: 
$$
       u  = \pm\; \sqrt{z\,A\, B}
$$
Also, recall that the spinor depends on the order of the circles: 
 if $u$ is a spinor for $(A,B)$,  then the spinor for  $(B,A)$  is 
 the spinor conjugated $u' = \pm iu$ .
For the geometric motivation of this definition see the integral 
example shown in Figure \ref{fig:process}.
For more details, see Section \ref{sec:context}.

\begin{theorem}[\bf spinor curl] 
\label{thm:curl}
\ Let $C_1$, $C_2$, and $C_3$ be three mutually tangent circles.  
Then the signs of three spinors of tangency can be chosen so that 
\begin{equation}
\label{eq:curl}
u_1  +  u_2  +  u_3  =  0  \qquad\quad         [ \hbox{\rm curl}\ \mathbf u = 0 ]
\end{equation}
\end{theorem}

\begin{figure}[h]
\centering
\includegraphics[scale=.8]{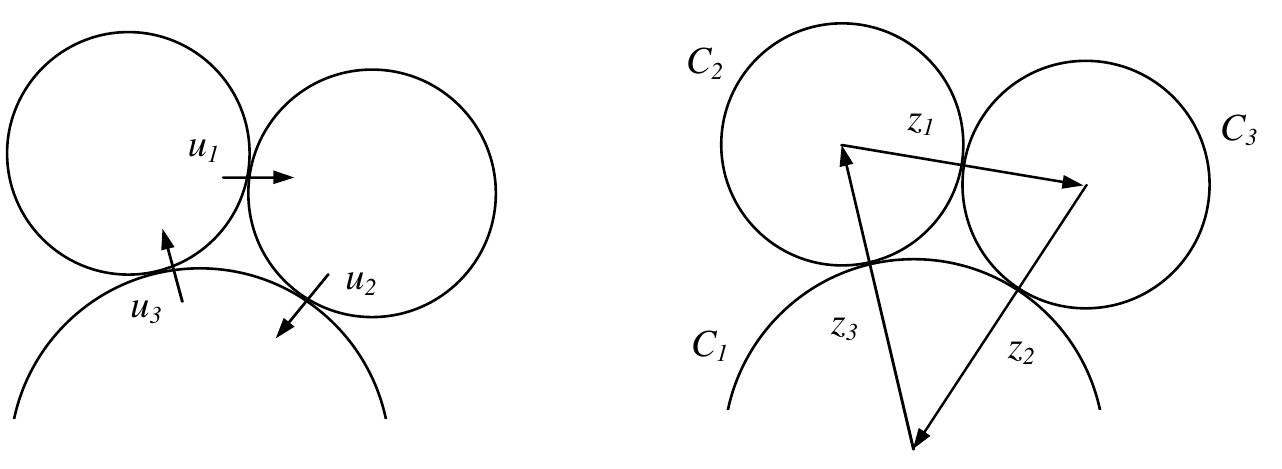}
\caption{Spinor Curl Theorem and its proof}
\label{fig:thmcurl}
\end{figure}

\begin{proof} 
Let $z_1$ , $z_2$ , and $z_3$  be three complex numbers representing 
the vectors joining the centers of the circles (see Figure \ref{fig:thmcurl} right).  Thus
$$
z_1  + z_2  + z_3  =  0
$$
Denote: 
$$
a = |z_1| = r_2 + r_3,  \quad b = |z_2| = r_3 + r_1,  \quad  c = |z_3|  =  r_1 + r_2 .
$$
The radii and curvatures of the circles are determined by these values:
\begin{equation}
\label{eq:defofr}
\begin{array}{rll}     
             r_1 =  (b + c - a)/2    &\quad\hbox{and}\qquad&       	C_1 = 1/r_1 \\
		r_2 =  (c + a -b)/2   &\quad\hbox{and}\qquad&               C_2 = 1/r_2\\
             r_3 =  (a + b - c)/2   &\quad\hbox{and}\qquad&	        C_3 = 1/r_3     
\end{array}
\end{equation}
Following (\ref{eq:spinor}), spinors $u_1$, $u_2$ and $u_2$ are defined by
\begin{equation}
\label{eq:defofu}
\begin{array}{l}          
    u_1^2  \  = \  z_1\, C_2 C_3 \\[3pt]
    u_2^2  \  = \  z_2 \,C_3C_1 \\[3pt]
    u_3^2  \  = \  z_3 \,C_1C_2
\end{array}
\end{equation}
The claim is thus that  there exists a choice of values of $\varepsilon \in \{ -1, +1 \}$ such that
$$
      \varepsilon_1 u_1 + \varepsilon_2 u_2 + \varepsilon_3 u_3  =0
$$
Denote 
\begin{equation}
\label{eq:defofF}
F  =  (u_1 + u_2 + u_3) (u_1 + u_2 - u_3) (u_1 - u_2 + u_3) (-u_1 + u_2 + u_3)
\end{equation}
We need to show that $F$ vanishes.  Expanding the above 
and then using the definition of spinors   (\ref{eq:defofu}) leads to:
$$
\begin{array}{rl}
     F   &=  - u_1^4  - u_2^4 - u_3^4    +   2 u_1^2u_2^2  +  2u_2^2 u_3^2  +  2u_3^2u_1^2  \\[7pt]
          &=  - z_1^2 C_2^2 C_3^2  -  z_2^2 C_3^2 C_1^2  -  z_3^2 C_1^2 C_2^2  +  2 z_1 z_2 C_2 C_3^2 C_1  
                          +  2 z_2 z_3 C_3 C_1^2 C_2   +  2 z_3 z_1 C_1 C_2^2 C_3 \\[7pt]
          &=  - C_1^2C_2^2C_3^2 
                    \; \left[\, z_1^2/C_1^2  +  z_2^2/C_2^2  +  z_3^2/C_3^2  
                                            -  2z_1z_2 /C_1C_2  -  2z_2z_3/C_2C_3  - 2z_3z_1 /C_3C_1\, \right]
\end{array}
$$
But $C_i = 1/r_i$,  thus the part in the bracket may be rewritten:
$$
     F  \ = \ -C_1^2  C_2^2  C_3^2 \cdot   [       z_1^2  r_1^2  +  z_2^2  r_2^2  +  z_3^2  r_3^2  
                                                                -  2z_1  z_2  r_1  r_2 - 2 z_2  z_3  r_2  r_3 - 2 z_3  z_1  r_3  r_1 ]
$$
Substituting (\ref{eq:defofr}) to the above expression and grouping by the products of $a$, $b$, $c$, 
we arrive at:
$$
\begin{array}{rl}
     F  = - \frac{1}{4}\, C_1^2  C_2^2  C_3^2 \ \cdot \!\! &[ \; (a^2 + b^2 + c^2) (z_1 + z_2 + z_3)^2  + \\ [3pt]
                                                                 &\qquad  - 2ab \, [z_3^2 - (z_1 + z_2)^2]  \\[2pt]
                                                                 &\qquad  - 2bc \, [z_1^2 - (z_2 + z_3)^2] \\[2pt]
                                                                 &\qquad  - 2ca \, [z_2^2 - (z_3 + z_1)^2]   \; ]  
\end{array}
$$
One may easily factor out the sum $(z_1 + z_2 + z_3)$:
$$
\begin{array}{rl}
      F  &= -\frac{1}{4} C_1^2 C_2^2 C_3^2
                                 \cdot    \left[  (a^2 + b^2 + c^2) (z_1 +  z_2 +  z_3)^2  + \right. \\
                   &\qquad\qquad\qquad\qquad\qquad   + 2ab\ (z_1 +  z_2 +  z_3)(z_1 +  z_2 -  z_3)  \\
                   &\qquad\qquad\qquad\qquad\qquad   + 2bc\ (z_1 +  z_2 +  z_3)( z_2 +  z_3 - z_1)  \\
                   & \qquad\qquad\qquad\qquad\qquad  \left. + 2ca\ (z_1 +  z_2 +  z_3)( z_3 + z_1 -  z_2)\phantom{{}^{1}} \right]  \\[5pt]
          &= -\frac{1}{4} C_1^2 C_2^2 C_3^2   \  (z_1\! +\!  z_2\! +\!  z_3) \cdot 
                   \left[  (a\! +\! b \!+\! c)^2 \, (z_1\! +\!  z_2\! +\!  z_3)     - 4(abz_3 \!+\! bc z_1\!+\! caz_2) \, \right] \,.
\end{array}
$$
Thus the initial condition  $(z_1 +  z_2 +  z_3) = 0$ implies $F = 0$, which means that one of the factors of $F$ in (\ref{eq:defofF}) must vanish.  
\end{proof}

One may form a converse theorem, which has the spirit of a simple theorem on complex numbers.
It may be stated without the geometric context of circles or disks
and it does not play any role in the rest of this paper.

\begin{theorem}[\bf converse of curl theorem]
\label{thm:curlconverse}
Let $u_1$,  $u_2$, and $u_3$ be three complex numbers such that
\begin{equation}
\label{eq:thm3a}
                                   u_1  +   u_2  +   u_3  =  0
\end{equation}
Then
\begin{equation}
\label{eq:thm3b}
            u_1^2\,g( u_2,  u_3) +  u_2^2 \,g( u_3, u_1) +  u_3^2\, g(u_1,  u_2)   =  0
\end{equation}
\end{theorem}

\begin{proof}
Starting with the left-hand-side expression of (\ref{eq:thm3b}) multiplied by $2$, we have:
$$
\begin{array}{l}
      u_1^2 \,      (       u_2 \bar u_3 +  u_3 \bar u_2)  
                             +   u_2^2 ( u_3 \bar u_1 + u_1 \bar u_3)  
                             +   u_3^2 (u_1 \bar u_2 +  u_2 \bar u_1)   \\[3pt]
      \qquad\quad    =        u_1 u_2 \bar u_3 \, (u_1 +  u_2)  
                                   +   u_3u_1 \bar u_2\, ( u_3 + u_1)  
                                   +   u_2 u_3 \bar u_1\, ( u_2 +  u_3) \\[3pt]
      \qquad\quad      =          u_1 u_2 \bar u_3 (- u_3)  
                                        +   u_3u_1 \bar u_2 (- u_2 )  
                                        +   u_2 u_3 \bar u_1(-u_1) \\[3pt]
      \qquad\quad      =  - u_1 u_2 u_3 \; ( \bar u_3 +  \bar u_2 + \bar u_1)  \\[3pt]
      \qquad\quad      =   0 \, ,
\end{array} 
$$
as stated. 
\end{proof}


It may be somewhat surprising that spinors preserve information about the sizes of circles in a rather simple way:

\begin{theorem} 
\label{thm:curvs}
 If $u$ is the tangency spinor for two tangent circles of curvatures $A$ and $B$  then 
\begin{equation}
\label{eq:thm3}
                     \|\,u\,\|^2 = A + B
\end{equation}
\end{theorem}

\begin{proof}
Starting with (\ref{eq:defofu}) we have $| u_2|^2  =  |z||AB|  =  |r_1 + r_2|\/AB  = (1/A + 1/B) AB  =  A + B$.  
The signs for inner and outer tangency are easy to verify.  
\end{proof}

\begin{figure}[H]
\centering
\includegraphics[scale=.8]{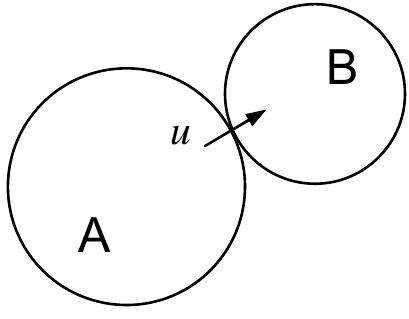}
\caption{Spinor and the curvatures }
\label{fig:u}
\end{figure}

\begin{theorem}[\bf curvatures from spinors]    
\label{thm:curv}
In the system of three mutually tangent circles, the symplectic product of two spinors directed outward from (respectively inward into) one of the circles equals (up to sign) its curvature, e.g., following notation of Figure \ref{fig:thm4}:
\begin{equation}
\label{eq:thm4}
           C \ = \  \pm\; a\times b 
\end{equation}
\end{theorem}

\begin{figure}[h]
\centering
\includegraphics[scale=.8]{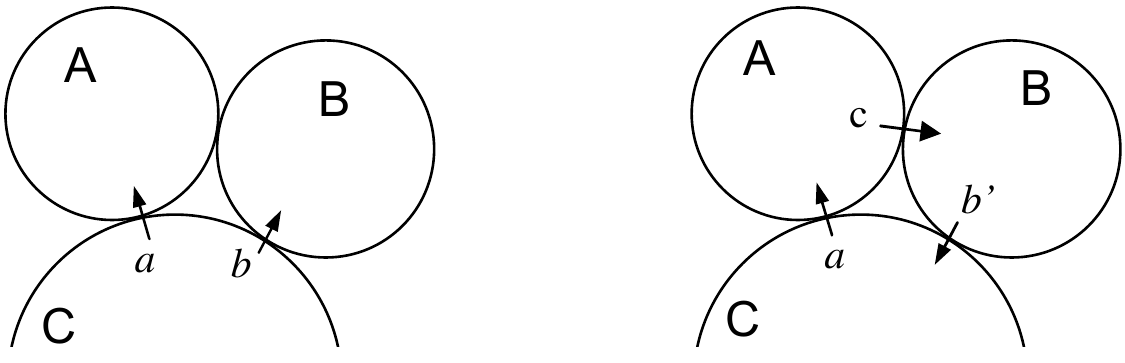}
\caption{Left: Theorem \ref{thm:curv};  Right: its proof}
\label{fig:thm4}
\end{figure}

\begin{proof}
Set three spinors for the three circles in a circular way as in Theorem \ref{thm:curl} 
so that $a + c + b' = 0$ (see Figure \ref{fig:thm4} right).  Consider, 
$$
\begin{array}{rl}
                           \|\, c\,\|^2   &= \  (\,a+b'\,)\bigcdot (\,a+b'\,) \\[7pt]
                                       &=  \  \|\, a\,\|^2 + 2\,a\bigcdot b' + \|\,b'\,\|^2 
\end{array}
$$
Using Proposition \ref{thm:curvs} we get  
$$
A + B  =  (A + C) + 2\, a\bigcdot b' + (B + C) 
$$
Thus 
$$
C = -  a\bigcdot b' = - a \times b = b\times a
$$
due to (\ref{eq:gomega1}).   
Ambiguity of the sign in \eqref{eq:thm4} originates from ambiguity of each of the spinors $a$ and $b$.
For some rectification see Section \ref{sec:pm}.
\end{proof}

Now we move to our main theorem, the theorem that involves all four circles in a Descartes configuration.  
It may metaphorically be called a ``square root of Descartes theorem'': 
\\

\begin{theorem} [\bf The Fundamental Theorem for Tangency Spinors]  
\label{thm:div}
Let $A$, $B$, $C$, and $D$ be four circles in a Descartes configuration.  
\\[3pt]
\begin{addmargin}[1em]{0em}   
{\bf Version A [vanishing divergence]:}  
If $a$, $b$ and $c$ are tangency spinors for pairs $AD$, $BD$ and $CD$  
(see Figure \ref{fig:thmdiv} left),
then their signs may be chosen so that 
\begin{equation}
\label{eq:thm5a}
 					a + b + c = 0 	\qquad	[\hbox{``\rm div}\, \mathbf u = 0'']                                   
\end{equation}
The property holds for both the inward and the outward oriented spinors.
\\[7pt]
{\bf Version B [additivity]:}  If $a$ and $b$ are spinors of tangency for pairs $CA$ and $CB$ 
(see  Figure \ref{fig:thmdiv} right), 
then there is a choice of signs so that the sum 
\begin{equation}
\label{eq:thm5b}
                                      c = a + b
\end{equation}
is a spinor of tangency for $CD$.  
\end{addmargin}
\end{theorem}

\begin{figure}[h]
\centering
\includegraphics[scale=.8]{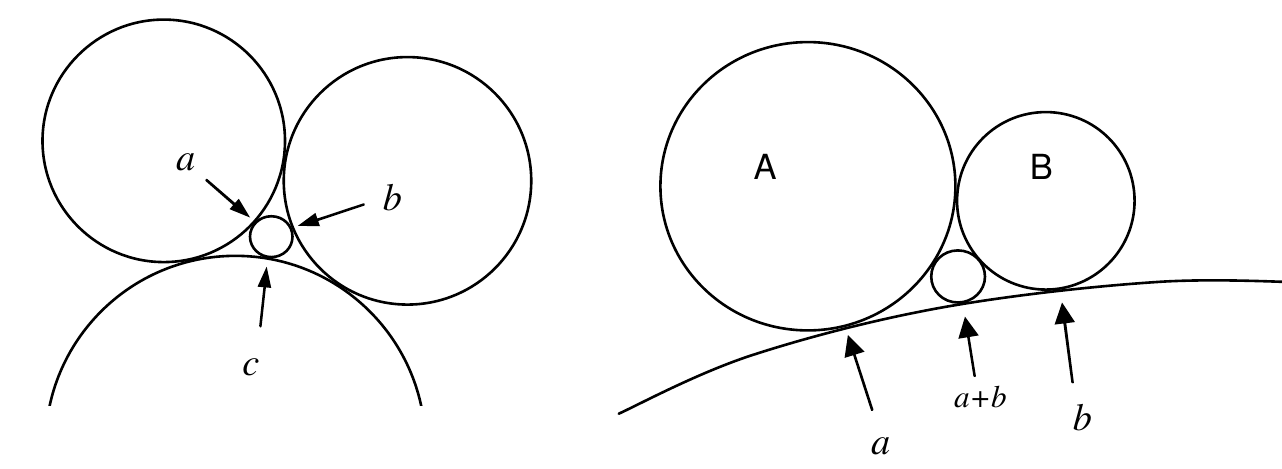}
\caption{(a) vanishing divergence, (b) spinor addition}
\label{fig:thmdiv}
\end{figure}

\begin{proof}
Let $A$, $B$, $C$ and $D$ be four pair-wise tangent circles. 
Let $a$ and $b$ be spinors of tangency for pairs $CA$ and $CB$, respectively.  
We may assume that curvature of one of the circle does not vanish, say $C \not=  0$.
By virtue of Theorem~\ref{thm:curv}, 
we have the values of symplectic products:
$$
\begin{array}{cl}
 			  (i) & \quad   a\times b \ = \ \pm \,C \\
 			 (ii) & \quad   c\times a  \ =  \ \pm \,C \\
 	 		(iii) & \quad   c\times b  \ =  \ \pm  \,C
%
\end{array}
$$
Spinor $c$ for pair $CD$ must be a linear combination of $a$ and $b$, 
say $c = pa + qb$, for some $p, q \in \mathbb R$. 
Substitute it to ({\it ii\/}): 
$$
  \pm\, C  \ =\   c\times a \ =\  (p a + q b)\times a  
                       \ = \  p( a\times a )+ q (b\times a)  
                      \  = \ q\,(b \times a) 
$$
Using now ({\it i\/}), we get   $C   = \pm q C$, or $q=\pm 1$.  
Similarly, applying this argument to ({\it iii\/}) implies $p=\pm 1$..
Thus $c = \pm a  \pm b$,  and what remains is to choose the signs 
of spinors to get the claims of each of the two versions of the theorem.    
\end{proof}

\begin{figure}[H]
\centering
\includegraphics[scale=.7]{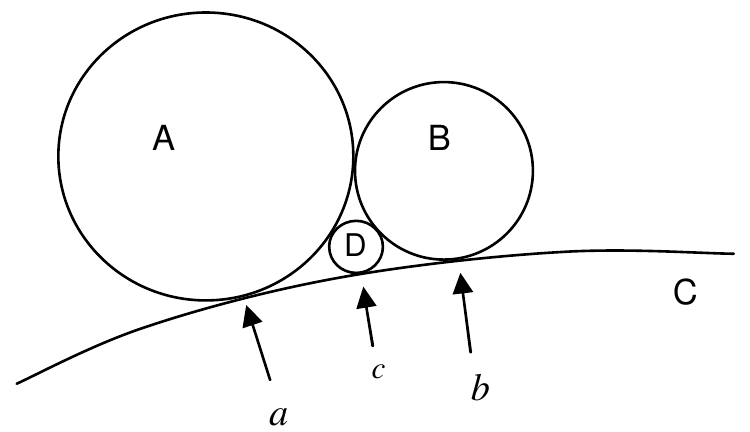}
\caption{Proving Theorem \ref{thm:div}}
\label{fig:provediv}
\end{figure}

\noindent
{\bf Corollary:}
Theorem \ref{thm:div}(B) can be iterated to produce spinors for all circles 
inscribed between two initial circles, see Figure \ref{fig:spiniterate}. 
Inspect the spinors in the disk of curvature 2 or the spinors around the great external circle 
in Figure \ref{fig:AWspinors}.
It has the flavor of Ford's result for relation between fractions and circles drawn on the number axis \cite{For}.
This will be the topic of the subsequent paper.

%

\begin{figure}[h]
\centering
\includegraphics[scale=.8]{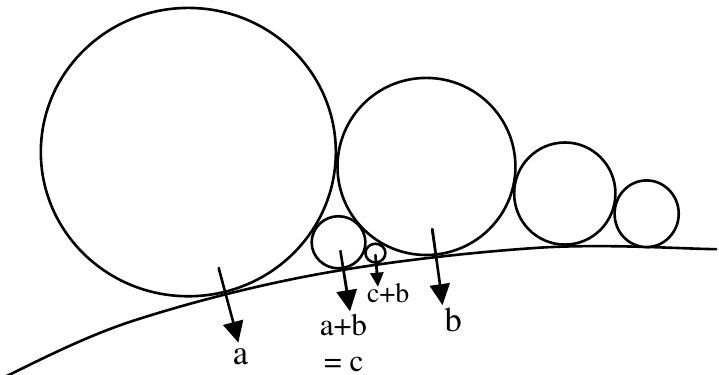}
\caption{Spins to a circle form a Stern-Brocot tree}
\label{fig:spiniterate}
\end{figure}

%


\subsection{\bf\large From spinors to Descartes Theorem}  

The Fundamental Tangency Spinor Theorem \ref{thm:div} contains the Descartes theorem on four tangent circles.
Interestingly, the proof -- although simple -- is not eminently visible.  
Simple squaring the formula leads to nowhere
(the reader is encouraged to try it before reading further). 

\begin{theorem} 
\label{thm:proof}
The Fundamental Tangency Spinor Theorem implies the Descartes Theorem
\end{theorem}

\begin{proof}
We shall use the well-known relation
\begin{equation}
\label{thm:calc3}
|\, \mathbf u\bigcdot \mathbf v \,|^2 \ + \ |\, \mathbf u\times \mathbf v \,|^2 \ = \  \|\,\mathbf u\,\|^2\|\, \mathbf v\,\|^2
\end{equation}
which is the Pythagorean Theorem $\cos^2\theta + \sin^2\theta = 1$ in disguise, 
known for 3D vector calculus and valid in 2D.  
Let us start with configuration and notation as in Figure \ref{fig:provediv}.
$$
a+b=d \quad\Rightarrow\quad  \|d\|^2 = \|a\|^2 + 2\,a\bigcdot b  + \|b\|^2 
$$
This gives the meaning of the inner product: 
\begin{equation}
\label{eq:P1}
 \begin{array}{rl}
 2\, a\bigcdot b &\ = \ \|d\|^2 - \|a\|^2 - \|b\|^2    \\[3pt]
               &\ = \ C+D-(A+C)  - B+C)        \\[3pt]
               &\ = \ D-A - B-C        
\end{array}                                                                              
\end{equation}
On the other hand, we may calculate the scalar product using \eqref{thm:calc3}:
\begin{equation}
\label{eq:P2}
\begin{array}{rl} 
|a\bigcdot b|^2 &\ =\  \|a\|^2 \|b\|^2 -(a\times b)^2\\[3pt]
                   &\ = \ (A+C )(B+C)-C^2\\[3pt]
                   &\ = \ AB + BC + CA
                   \end{array}
\end{equation}
Now, equating the scalar product of two equations, we get
\begin{equation}
\label{eq:last}
4(AB +BC +CA) ^2 = (D-A-B-C)^2
\end{equation}
This, after expanding and regrouping, is the Descartes Formula \eqref{eq:Descartes1}.
\end{proof}

\noindent
{\bf Remark:} The above theorem and its proof may be seen as an alternative proof of Descartes Theorem.
\\

\noindent
{\bf Remark:} Notice that without much ado one gets the ``solution'' \eqref{eq:Descartes1sol}
 to the quadratic equation directly from the last equation of the proof:
$$
D \ = \ A+B+C \ \pm \ 2\sqrt{AB +BC +CA}
$$

~

Since the cross product of two spinors from $C$ to $A$ and $B$ has a geometric meaning 
of the curvature of the disk $C$, 
a natural question arises if the dot product has a meaning too.

\begin{theorem}[\bf Mid-circle from spinors]     
\label{thm:inner}
In the system of three mutually tangent circles, the scalar product of two spinors directed outward from 
(or inward to) one of the circles to the other two equals (up to sign) 
the curvature of the circle that passes through the points of tangency between the three circles
(see Figure \ref{fig:thm5}):
\begin{equation}
\label{eq:thm5}
           a\bigcdot b \ = \ \pm K
\end{equation}
\end{theorem}

\begin{proof}
Choose the signs of the spinors between the disks so that
$$
a+b=c
$$
Norm-squaring both sides gives result that appeared in \eqref{eq:P2}
$$
a\bigcdot b \ = \  \pm \sqrt{AB+BC+CA} \,,
$$
which is the well-known formula for the mid-circle,
a circle that is orthogonal to each of them, $K\bot A$, $K\bot B$, and $K\bot C$.
\end{proof}

\begin{figure}[h]
\centering
\includegraphics[scale=.8]{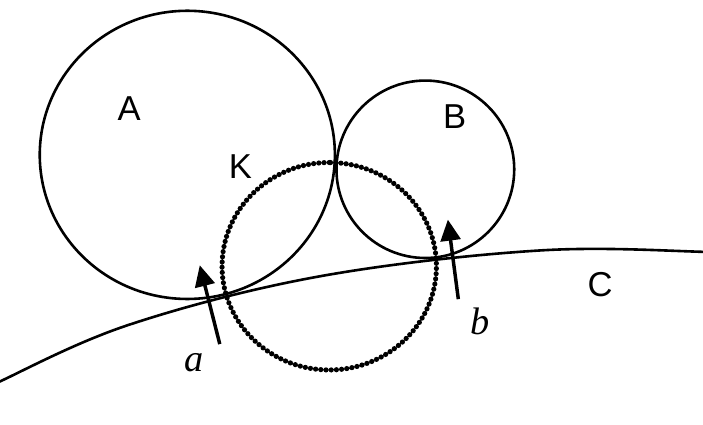}
\caption{Spins to a circle form a Stern-Brocot tree}
\label{fig:thm5}
\end{figure}

\noindent
{\bf Corollary:}  Equation \eqref{eq:last} in the proof of Theorem \ref{thm:proof} 
provides another formula for the mid-circle
$$
K \ = \  \pm\,\frac{1}{2}(A+B+C-D)\,,
$$
where $D$ is any of the two solutions of the Descartes problem for $A$, $B$. and $C$.

\subsection{\bf Topological obstructions}
\label{sec:pm}

The fact that the Descartes formula provides two solutions for a given triplet of mutually tangent disks
is in agreement with the sign ambiguity of spinors.
We shall now ``tame'' this ambiguity.
\\

\noindent
{\bf Definition:}  Referring to Figure \ref{fig:parallel} left, 
we say that signs of spinors $a$ and $b$ are {\bf harmonized} over $D$ (or spinor $c$)
if 
one may choose the sign of $d$ so that 
\begin{equation}
\label{eq:ab}
 a+b =d
 \end{equation}
(as opposed to $a-b$). 
In such a case we shall say that the sign of $a$ is {\bf parallelly transported} to $b$
along the arc of $C$ that passes through the point of tangency 
of $C$ with $D$ (in short: over the arc of $C$  through  $D$).

~

\noindent
{\bf Corollary:}
If spinors in Theorem \ref{thm:curv}, are harmonized 
in the counter-clock direction of the disk they leave 
then  the Formula \eqref{eq:thm4} of 
becomes  
$$
 C \ = \  +\, b\times a 
$$
(counter-clock convention).  
Changing the sign of any of the two spinors (or, equivalently, changing of the order in the product)
corresponds to harmonizing over the complementary arc through the other Descartes solution, $D'$, to the Descartes problem 
for the triplet $A,\;B,\;C$.
\\

\begin{figure}[h]
\centering
\includegraphics[scale=.8]{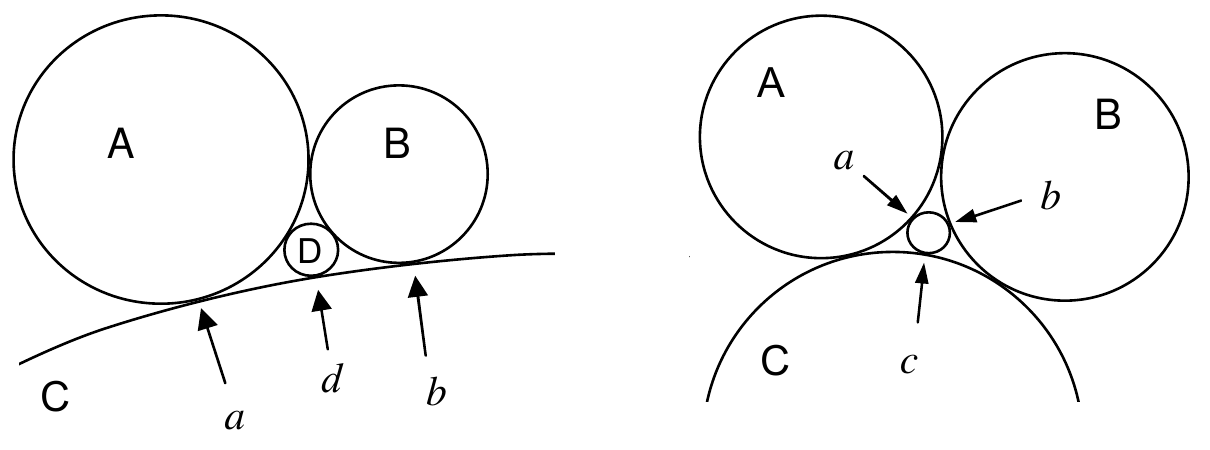}
\caption{Left:  Definition of parallel transport of signs.  Right: Impossibility of harmonizing the spinors
in the Descartes configuration}
\label{fig:parallel}
\end{figure}

The signs of the spinors in an Apollonian gasket cannot be arranged globally 
so that the properties of the Fundamental Theorem hold for each triple of mutually
tangent disks.
The problem starts already with Descartes configuration.
Indeed, consider the Descartes configuration as the one in Figure \ref{fig:parallel}, right.
Suppose one may harmonize the spinors $a$, $b$ and $c$
along the circle in the middle, clockwise.
We can do it in three steps: 
$$
\begin{array}{crl}
(i) & \quad\hbox{transport the sign from $a$ to $c$ over $B$:}  &\quad  b = a+c \\[3pt] 
(ii) & \quad \hbox{transport the sign from $c$ to $b$ over $A$:}  &\quad  a =b+c \\[3pt] 
(iii)&  \quad \hbox{transport the sign from $b$ to $a$ over $C$:}  &\quad  c=  a+b 
\end{array}
$$
Now, add ({\it i}\/) and ({\it iii}\/) side-wise:
$$
              b+c = b+ c + a+a
$$
which would make sense if one of the $a$'s were a negative of the other.  
This is an indication that transfer of $a$ around the inner circle changes its sign.
Performing the transfer around twice would restore the original sign.
\\
 
This imitates the spinors describing the spin of electrons, which change the sign after a single rotation of electron.
The phenomenon corresponds to the double covering of the group SU(2) over SO(3), 
while the group-theoretic analogue for tangency spinors corresponds to the double cover
$$
\SL(2,\mathbb R)  \ \xrightarrow{\quad 2:1\quad} \ \SL(2,,\mathbb R)  \,.
$$

For numerical examples, consult Figure \ref{fig:AWspinors}.  
The spinors in the disk of curvature 2 are harmonized.
But extending the list of the spinors to the bottom part brings surprise in change of sign of 
the spinor at the starting point 
But this cannot be extended to the whole circle.

\section{Context and clarifications}  
\label{sec:context}

In this section we review some facts about geometry of circles, Apollonian disk packings and Pythagorean triples
in order to provide a wider context to the results of the previous section,
 and to hint towards some generalizations.

\subsection{\bf Apollonian Window, Pythagorean triples, and Euclidean parametrization}

Integral Apollonian disk packings 
attract considerable attention due to their rich number-theoretical content
\cite{Hir,GLM1,GLM4,LMW,Man,Mel}.
Figure~\ref{fig:S32} shows one that is particularly graceful -- the {\bf Apollonian Window};
denote it $\mathcal A$.
{\bf Symbols} that label some of the circles have the following meaning:
curvatures (reciprocals of radii) are indicated in the denominators 
while the positions of the centers may be read off by interpreting the symbol 
as a pair of fractions \cite{jk-c}.
For example, here is how you decode one of them:
$$
\hbox{\sf symbol:\  } 
       \frac{3,\;4}{6}  
                         \qquad \Longrightarrow \qquad 
       \begin{cases}     \hbox{\sf radius:\  } &  r = \frac{1}{6} \\
                                  \hbox{\sf center:\ }  & (x,\,y)  =   \left(\frac{3}{6},\, \frac{4}{6}\right)
                                                                    =   \left(\frac{1}{2},\, \frac{2}{3}\right)
       \end{cases}
$$
The two numbers in the numerator will be called the {\it reduced coordinates} 
of the a disk's center, typically denoted by dotted symbols, as $(\dot x, \, \dot y)$.
What makes the Apollonian Window special is the fact that all of its symbols have integer entries.

\begin{figure}[h]
\centering
\includegraphics[scale=.8]{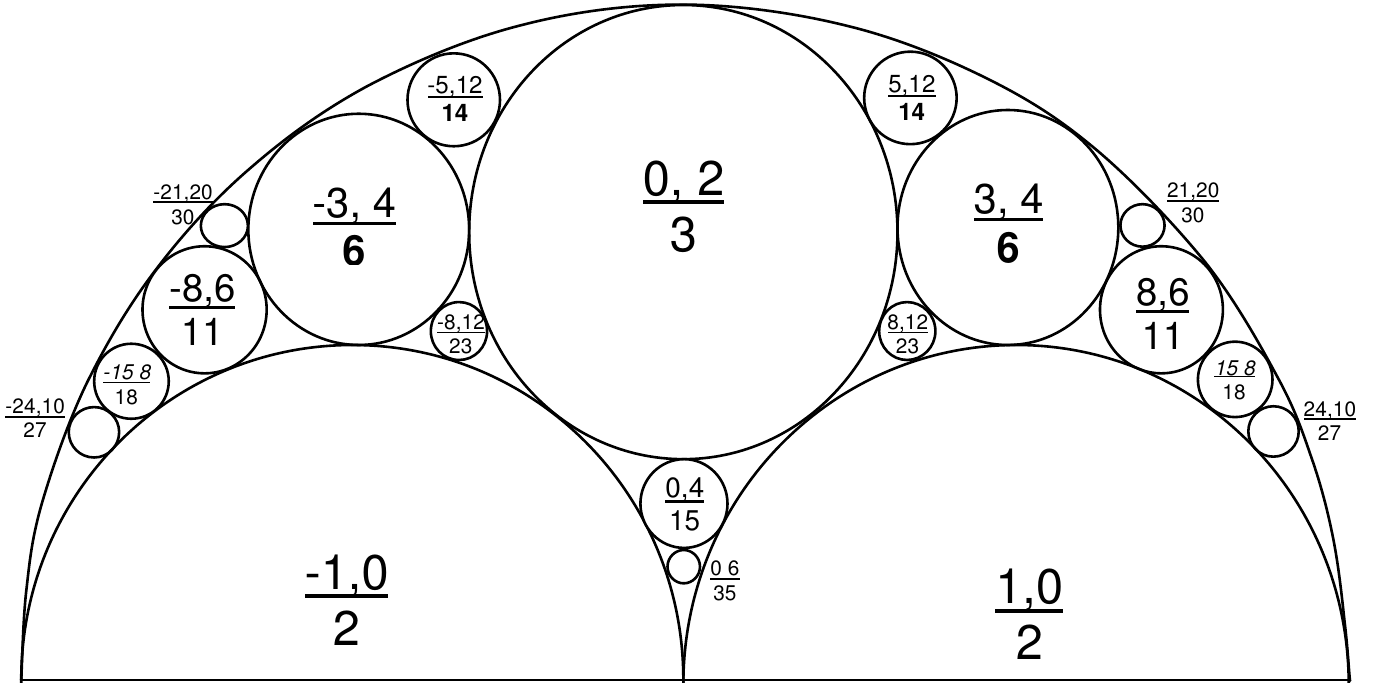}
\caption{Symbols in the Apollonian Window.  Only the upper half is shown.}
\label{fig:S32}
\end{figure}

The key tool for creating the data for the disks is the linearized version
of Descartes'  Formula \eqref{eq:Boyd1} :
\begin{equation}
\label{eq:Boyd}
                 D+D' = 2(A+B+C)
\end{equation}
Interestingly, the same formula holds for the corresponding reduced 
coordinates $(\dot x, \dot y)$
(consult \cite{jk-d}).
Since an Apollonian gasket is a fractal completion of a Descartes configuration 
and all symbols may be derived with (\ref{eq:Boyd}) given the first four, 
integrality of the symbols in $\mathcal A$ follows from the integrality of the initial four 
disks (or actually any four disks in Descartes configuration in the gasket).

\begin{figure}[h]
\centering
\includegraphics[scale=.8]{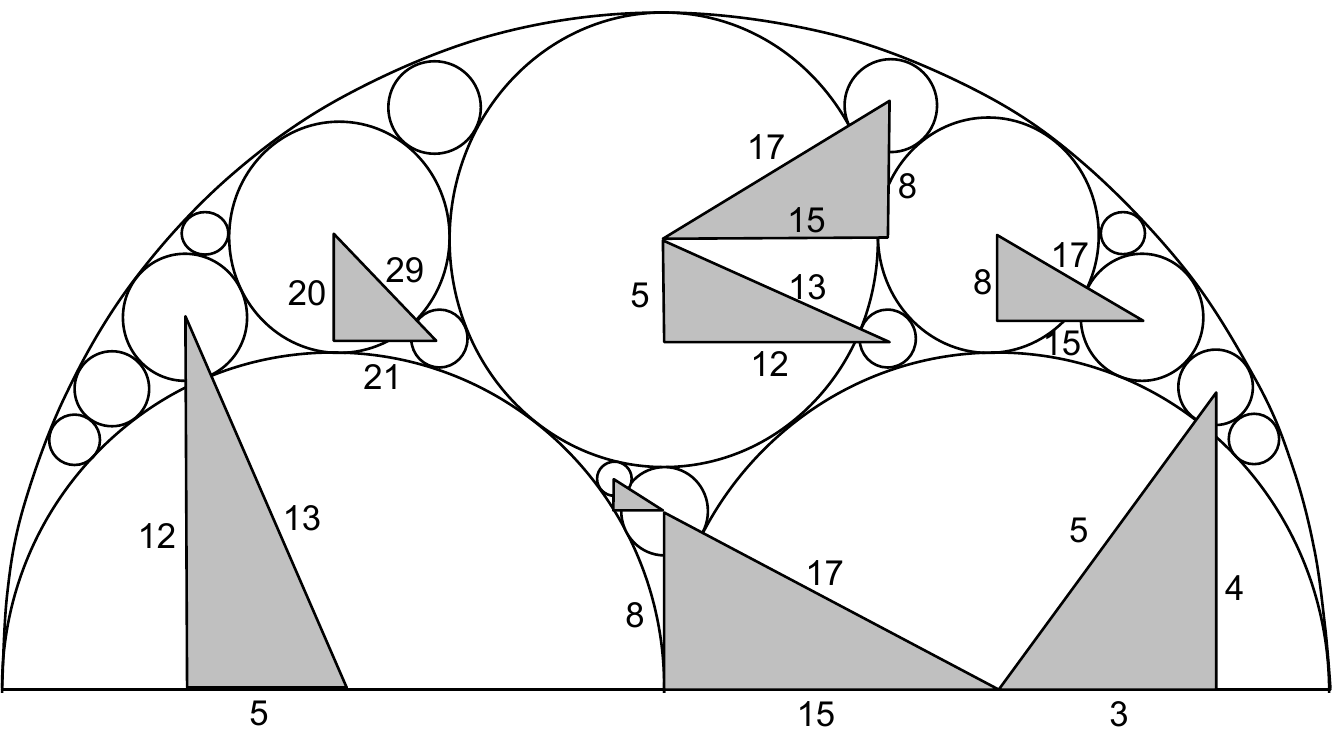}
\caption{Pythagorean triples in the Apollonian Window}
\label{fig:AWtriangles}
\end{figure}

Every pair of tangent circles in $\mathcal A$ defines a triangle
with sides proportional to a Pythagorean triple as follows:
\begin{equation}
\label{eq:product}
     \frac{\dot x_1, \;  \dot y_1}{ \beta_1}
\ \Join  \
     \frac{\dot x_2,\; \dot y_2}{ \beta_2}
\quad = \quad
       \left[\begin{matrix}
                     a \\
                     b \\
                    c
     \end{matrix}\right]
\equiv 
  \left[\begin{matrix}
                     \beta_1 \dot x_2 - \beta_2 \dot x_1 \\
                     \beta_1 \dot y_2 - \beta_2 \dot y_1 \\
                    \beta_1+\beta_2
     \end{matrix}\right]
\end{equation}
where $a^2+b^2=c^2$ is easy to check. 
Figure \ref{fig:AWtriangles} shows some of the triples in $\mathcal A$.
Note that the Pythagorean triples \eqref{eq:product} are integral.
The actual triangle in the figure has its size scaled down by the factor of $\beta_1\beta_2$.

The next step is to recall that Pythagorean triangles admit Euclidean parameters 
\cite{Sie,TT} that determine them via: 
$$
u=[m,n] \quad \mapsto \quad (a,b,c) = (m^2-n^2,\ 2mn, \ m^2+n^2)
$$
As explained in  \cite{jk-c}, Euclidean parameters can be viewed as a {\it spinor}.
Indeed, a Pythagorean triple, satisfying $a^2+b^2-c^2=0$, may be viewed as a null-vector
of Minkowski space $\mathbb R^{2,1}$ and
as such, it may be represented as the tensor square of a spinor 
from the associated two-dimensional spinor space $\mathbb R^2$.  
Fortunately, for the purposes of this paper it suffices to represent the spinor 
as a complex number $u\in\mathbb C\cong\mathbb R^2$
 via identification $[m, n] \equiv m+ni$. 
The corresponding Pythagorean triple is defined by squaring:
$$
            u=m+ni \quad \to \quad u^2  \ = \  (m^2-n^2) + 2mn\, i  \ = \  a+bi   
$$
with $c=|u^2|=m^2+n^2$.  Clearly, the spinor is defined up to a sign, since $(-u)^2=u^2$.
\\

\begin{figure}[H]
\centering
\includegraphics[scale=.7]{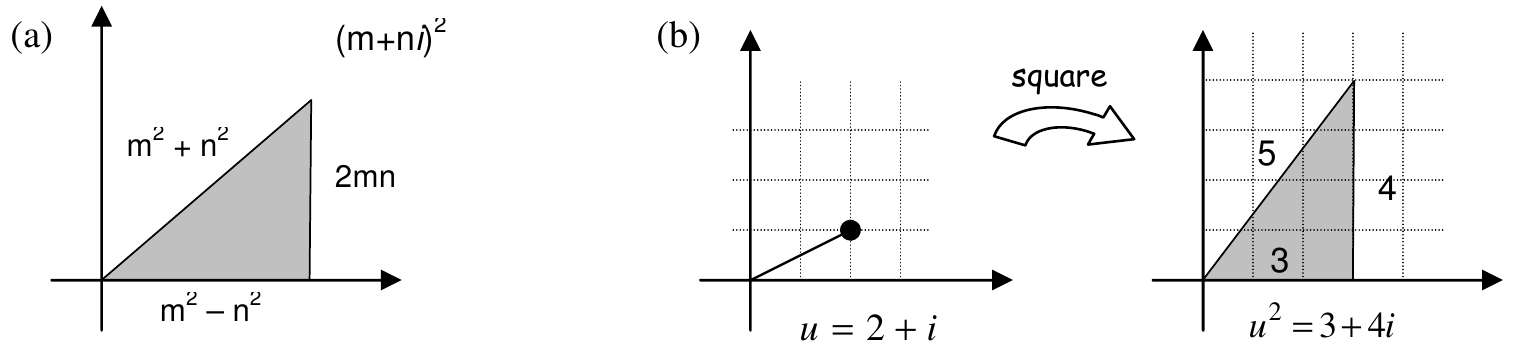}
\caption{\small Euclidean parameterization as a spinor}
\label{fig:square}
\end{figure}

And this is the initial inspiration for introducing the concept of tangency spinors.
Reconstructing a number of spinors via these Pythagorean triples, as in Figure  \ref{fig:AWspinors}
reveals a number of  features visually detectable by inspection;
realization that they start with the Descartes configuration lead to the present paper.


~\\


\begin{figure}[h]
\centering
\includegraphics[scale=.8]{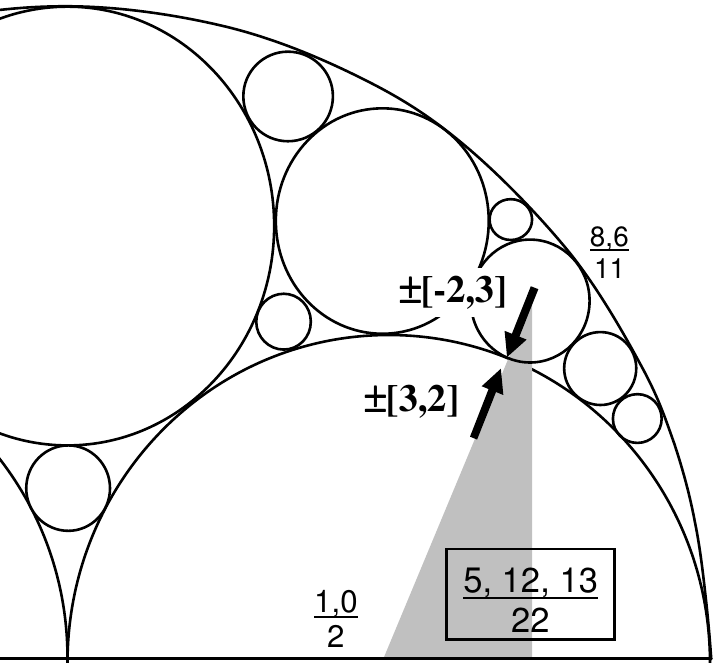}
\caption{Spinors in the Apollonian Window}
\label{fig:21}
\end{figure}
\subsection{\bf Minkowski space}

By $\mathbb R^{n+1;1}$ we denote Minkowski space with a standard {\bf isotropic} basis, 
that is Cartesian space $\mathbb R^{n+2}$ with column vectors 
\begin{equation}
\label{eq:vectorsdot}
\mathbf v\ = \
   \left[\begin{matrix}
   \xi_1\\
   \xi_2\\
   \vdots\\
   \xi_n\\
   \beta\\
   \gamma
   \end{matrix}\right]_{\bbigcdot}
\ = \
   \left[\begin{matrix}
   \bm\xi \\
   \beta\\
   \gamma
   \end{matrix}\right]_{\bbigcdot}
\end{equation}
and with the norm defined by
$$
|\mathbf v |^2 = -\xi_1^2-\ldots -\xi_n^2 +\beta\gamma  
$$
The subscript dots in \eqref{eq:vectorsdot} indicate that the isotropic basis in use.
The matrix associated with this quadratic form is
$$
g=
\left[\begin{matrix}
-\mathbf{I} & 0 & 0\\
0 & 0 & 1/2\\
0 & 1/2 & 0\\
\end{matrix}\right]
$$
and Minkowski scalar product of two vectors $\mathbf v$ and $\mathbf v'$ is
\begin{equation}
\label{eq:inner}
         \langle \mathbf v ,\; \mathbf v' \rangle 
                    = \mathbf v^T  g\, \mathbf v' 
                    = -  \xi_1\xi'_1 \ \dots -  \xi_n\xi'_n 
                        +\frac{1}{2}\; \beta\gamma' + \frac{1}{2}\;  \beta'\gamma
\end{equation}
(This is a space of signature (n+1,1) as a simple change of basis would show).
We preserve the symbol $\mathbb R^{m,1}$ ({\it comma} versus {\it semicolon}) 
for the Minkowski space with the standard {\bf orthonormal} basis in which the metric given by matrix
$$
g=
\left[\begin{array}{rr}
-\mathbf I & \;0 \; \\
0 & \;1 \;
\end{array}\right]
$$

Let us introduce a nonlinear ``cross-product'' in the Malinowski space:

\begin{proposition}
\label{thm:basic}
Define a product 
$$
      \Join: \  \mathbb R^{n+1;1} \times \mathbb R^{n+1;1} \longrightarrow \ \mathbb R^{n+1;1}
$$
as 
$$
\mathbf v=
     \left[\begin{matrix}
        \bm\xi\\
       \beta\\
       \gamma
      \end{matrix}\right]_{\bbigcdot}
\  , \ 
\mathbf v'=
     \left[\begin{matrix}
        \bm \xi'\\
        \beta'\\
        \gamma'
      \end{matrix}\right]_{\bbigcdot}
\qquad \longrightarrow \qquad
      \mathbf v\Join \mathbf v'
      =           \left[\begin{matrix}
                      \beta \bm\xi' - \beta' \bm\xi\\
                      \beta+\beta'\\
                      \beta+\beta'
     \end{matrix}\right]_{\bbigcdot}
$$
Then if  \  $|\mathbf v|^2 = |\mathbf v'|^2=-1$ and $\langle \mathbf v, \mathbf v'\rangle = 1$ 
then $\mathbf v\Join \mathbf v'$ is a nul vector, $|\mathbf v\Join \mathbf v'|^2 = 0$.
\end{proposition}

\begin{proof}
Direct calculations.
\end{proof}
The equality of the isotropic components imply that 
the image of this map lies in the subspace of $\mathbb R^{n+1;1}$ 
that is isomorphic to the standard $(n+1)$-dimensional Minkowski space $\mathbb R^{n,1}$, 
thus the above map may be restricted to 
$$
       \mathbb R^{n+1;1} \times \mathbb R^{n+1;1}  \ \longrightarrow \ \mathbb  R^{n,1}_0 
$$
where $\mathbb R^{n,1}_0$ denotes the cone of null vectors in $\mathbb R^{n,1}$.
As is well known in the theory of representations of Clifford algebras \cite{Por,Por2,Lou}, 
the null vectors may be represented as "tensor squares"
of the spinors in the spinor space $\mathbb S$ corresponding to a spin representation of $\SO(n,1)$: 
$$
v = u^*\otimes u 
$$
where the star denotes an appropriate Hermitian conjugation in $\mathbb S$.  
(Consult \cite{Por} for the constructive way to determine $\mathbb S$ for a particular dimension).
More precisely, there exist a natural bilinear map $B:\mathbb S\otimes\mathbb S \to \mathbb  R^{n,1}$ 
defined up to a real scalar. 
On the diagonal entries it takes values in the light cone  $R^{n,1}_0$.
Effectively, we have a composition of maps 
\begin{equation}
\label{eq:maps}
        \mathbb R^{n+1;1} \times \mathbb R^{n+1;1}  \ \to \ \mathbb  R^{n,1}_0 \ \to \ \mathbb S
\end{equation}
where the last map 
is defined up to normed factor $z\in \mathbb F$ , $|z|^2=1$, 
where $\mathbb F$ is the field/algebra  
underlying the spinor space, $\mathbb F = \mathbb R,\, \mathbb C,\, \mathbb H$
(real numbers, complex numbers, quaternions and octonions). 

Minkowski spaces $\mathbb R^{2,1}$, $\mathbb R^{3,1}$, $\mathbb R^{5,1}$ and $\mathbb R^{9,1}$
are particularly interesting \cite{Del}, since the corresponding spin groups are  
$\SL(2,\mathbb F)$, with $\mathbb F =  \mathbb R,\, \mathbb C, \,\mathbb H,\, \mathbb O$,
respectively,
and the spinors form two-dimensional spaces (modules) over corresponding algebra $\mathbb F$.  
The present paper explores disks in $\mathbb R^2$, which corresponds to $\SL(2,\mathbb R)$.
Other cases will be discussed in a future paper.

\subsection{Geometric interpretation: circles as vectors in Minkowski space} 


Our interest in Minkowski space comes from a well-known fact \cite{jk-d,Ped70} 
that disks in $n$-dimensional  Euclidean space 
may be represented by the unit space-like vectors in $\mathbb R^{n+1;1}$, 
namely a disk of radius $r$ centered at $\mathbf x=x_1,... , x_n$ becomes a vector
$$
\hbox{Disk}(\mathbf x,r) \qquad   \longrightarrow \qquad
\mathbf v = 
\left[\begin{matrix}
                 \dot {\mathbf  x} \\
                 \beta \\
                 \gamma
\end{matrix}\right]
\ \equiv \ 
\left[\begin{matrix}
                 \mathbf x / r\\
                 1/r \\
                \frac{\mathbf x^2   -r^2 }{r}
\end{matrix}\right]
$$
where $|\mathbf v|^2=  -1$. The inner product of $\mathbb R^{n+1;1}$ has a geometric interpretation:
$$
\langle \mathbf v,\, \mathbf v' \rangle  = 
\begin{cases}
\ \quad   \cos \varphi                                  &\quad\hbox{if disks boundaries intersect}\\[3pt]
\ \dfrac{d^2-r_1^2-r_2^2}{r_1r_2  }         &\quad\hbox{in general}
\end{cases}
$$
In the case of 2-dimensional disks discussed in the previous section 
the map and associated notation are as follows
$$
   \hbox{Disks}\;(\mathbb R^2 )
                \ \longrightarrow\ \mathbb R^{3;1}: \quad
   \hbox{disk}((x,\,y), r) \   \longrightarrow \ 
  \left[\begin{matrix}
                 \dot  x \\
                \dot y \\
                 \beta \\
                 \gamma
  \end{matrix}\right]
\ \equiv \ 
\left[\begin{matrix}
                  x / r\\
                  y/r \\ 
                 1/r \\
                \frac{x^2 +y^2   - r^2 }{r}
\end{matrix}\right]
\ \leftrightarrow \ 
\frac{\dot x,\, \dot y}{\beta}
$$
where the last ``fraction '' is a simpler presentation of the Minkowski vector, 
since it contains the complete information 
(the value of $\gamma$ may be calculated from the other three components because of normalization).

A vector and its negative represent two disks, inner and outer, sharing the same circle as a boundary.
In particular, in the Apollonian Window,  the great circle centered at $(0,0)$  is a boundary of an outer disk 
and therefore its radius is negative one, and so is its curvature, $\beta=-1$.  
\\

The appearance of Pythagorean triples in the Apollonian Window $\mathcal A$ may 
be formally thought of as a corollary to Proposition \ref{thm:basic}.
Namely, for any two tangent circles in $\mathcal A$, since their inner product is $1$, we have:   
$$
     \left[\begin{matrix}
        \dot x_1\\
       \dot y_1\\
       \beta_1\\
       \gamma_1
      \end{matrix}\right]
\ \bowtie \
     \left[\begin{matrix}
        \dot x_2\\
        \dot y_2 \\
        \beta_2\\
        \gamma_2
      \end{matrix}\right]
\quad = \quad
       \left[\begin{matrix}
                     \beta_1 \dot x_2 - \beta_2 \dot x_1 \\
                     \beta_1 \dot y_2 - \beta_2 \dot y_1 \\
                    \beta_1+\beta_2
     \end{matrix}\right]
\equiv 
       \left[\begin{matrix}
                     a \\
                     b \\
                    c
     \end{matrix}\right]
$$
with $a^2+b^2=c^2$. This is equivalent to (\ref{eq:product}).
%
%
%
%
Clearly, for an arbitrary pair of tangent circles the resulting triangle does not have to be integral.

\subsection{Euclid parametrization as a spinor}

Here we recall the spinor interpretation of Euclid's parametrization of Pythagorean triples
following \cite{jk-c}. 
View  a Pythagorean triple   as a vector of Minkowski space $\mathbb R^{2,1}$.
As such it may be represented by traceless matrices on which the group $\SL(2,\mathbb R)$ 
acts by conjugation as the way of representing $\SO(2,1)$.
But being a null vector, satisfying $a^2+b^2-c^2=0$, 
it may be represented as a tensor product of a spinor with itself, conjugated, 
$u\otimes u^*$, where $u = [m,n]^T$.
\begin{equation}
\label{eq:mn}
\left[\begin{matrix}
- b &   a+c \\
a-c & b
\end{matrix}\right]
\ \  = \ \
2\;\left[\begin{matrix}
 m\\
n
\end{matrix}\right]
\otimes
\left[\begin{matrix}
  -n & m
\end{matrix}\right]
\end{equation}
In general the group $\SL(2,\mathbb R)$ serves as the symmetry group covering $\SO(2,1)$.
The numbers $m,n$ coincide with Euclid's parametrization.
In particular, Eq. (\ref{eq:mn}) is a tensor version of the standard form of the Euclids' relation: 
$$
              (a,b,c) = (m^2-n^2,\ 2mn, \ m^2+n^2)
$$
In the present paper we represent the spinors  
by complex numbers via this identification
 $$
     \left[\begin{matrix}
        m\\
        n
      \end{matrix}\right]
= m+n i
$$
This representation comes about as follows:
Multiply both sides of (\ref{eq:mn}) on right by the ``anti-diagonal'' matrix to get
$$
\left[\begin{matrix}
 a+c & -b \\
b & a-c 
\end{matrix}\right]
\ = \
2 \;\left[\begin{matrix}
m^2 & -mn \\
mn & -n^2 
\end{matrix}\right]
$$
Now it is a matter of projection along direction 
$\left[\begin{smallmatrix}1&\phantom{-}0\\ 0&-1\end{smallmatrix}\right]$ 
onto
$\hbox{span}\,\{
\left[\begin{smallmatrix}1&0\\0&1\end{smallmatrix}\right], 
\left[\begin{smallmatrix}0&1\\1&0\end{smallmatrix}\right], 
\left[\begin{smallmatrix}0&-1\\0&\phantom{-}0\end{smallmatrix}\right]
\}$ 
in the linear space of matrices 
$M_2(\mathbb R)$ to obtain  
$$
\left[\begin{matrix}
 a & -b \\
b & a 
\end{matrix}\right]
\ = \ 
\left[\begin{matrix}
m^2 -n^2& -2mn \\
2mn & m^2-n^2 
\end{matrix}\right]
\ \equiv \ 
\left[\begin{matrix}
m& -n \\
n & \phantom{-}m 
\end{matrix}\right]^2
$$
which is recognizably the standard representation of complex numbers, corresponding to 
\begin{equation}
\label{eq:c}
   a+b\,i \  = (m+n\,i)^2.
\end{equation}
This starting point (\ref{eq:mn}) of this derivation justifies calling 
Euclids' parameters a {\it spinor}, while the resulting and well-known (\ref{eq:c})
used in the proofs of the previous section conveniently utilize the  
algebraic structure of complex numbers.
\\
\\
{\bf Remark:}  The derivation of \eqref{eq:c} 
as the spinor property has a better algebraic representation in terms 
the split quaternions that happen to coincide 
with the Clifford algebra of the Minkowski space  
$\mathbb R^{2,1}$  (c.f., \cite{jk-k});
this will be addressed elsewhere. 

%
%
%
%

\section{Summary}
Every ordered pair of disks in a Euclidean plane gives rise to a spinor $u\in\mathbb R^2$, 
an element of 2-dimensional symplectic space, defined up to sign.  
Algebraically, the process may be described as follows: 
we represent disks by space-like unit vectors in the Minkowski space $M=\mathbb R^{3,1}$.
A certain product of such vectors lies in the null cone of a subspace
isomorphic to Minkowski space $\mathbb R^{2,1}\subset M$.
Such vectors may in turn be represented by a tensor square of a spinor,
an element of the spinor space associated to $\mathbb R^{2,1}$.


\def\maly{\!\!\!\!\!\!}

$$
\qquad\quad
\begin{tikzpicture}[baseline=-0.8ex]
    \matrix (m) [ matrix of math nodes,
                         row sep=2.5em,
                         column sep=3em,
                         text height=3.8ex, text depth=3ex] 
   {
\maly\maly   \mathbb R^{3;1}   \maly\maly
      & \quad\mathbb R^{3;1} \!\times\! \mathbb R^{3;1} \quad
                    & \quad \mathbb R^{2,1}_0 \quad   
                           &\quad\mathbb C \quad   \\
 \maly \maly   \begin{matrix}\hbox{disks} \\ \hbox{in plane} \end{matrix}\maly\maly
                    & \quad \begin{matrix}\hbox{tangent} \\ \hbox{disks} \end{matrix}  \quad     
                    & \quad \begin{matrix}\hbox{right}\\ \hbox{ triangles}\end{matrix}\quad
                        &\quad\hbox{spinors}\quad\\
     };
    \path[-stealth]
        (m-1-2) edge node[above] {$\Join$} (m-1-3)
        (m-1-3) edge node[above] {s}  (m-1-4)
                       (m-2-1) edge node[right] {}  (m-1-1)
                       (m-2-2) edge node[right] {}  (m-1-2)
                       (m-2-3) edge node[right] {}  (m-1-3)
                       (m-2-4) edge node[right] {}  (m-1-4);
%
\end{tikzpicture}   
$$

The inspiration for this construction comes from the presence of the 
Pythagorean triangles in the Apollonian Window, a special case of an integral disk packing. 
Algebraic interpretation of this construction lies in the Euclidean parametrization of 
Pythagorean triangles.
But the results do not rely on integrality, even if the integral packings are 
especially interesting from the number-theoretic point of view.

The theorems describing the  behavior of spinors in a Descartes configuration of disks 
are is the main content of the paper.  
Among the rather intriguing properties are their mutual relations 
and relations to the disks' curvatures.
The ``spinor fiber bundle" over an Apollonian disk packing
and its topological aspects will be discussed in  a subsequent paper.

\section*{Acknowledgments}

The author is grateful to Philip Feinsilver for his interest in this work and his valuable comments.



\begin{thebibliography}{99}



\bibitem
{Boyd}   David W. Boyd,
An algorithm for generating the sphere coordinates in a three-dimensional osculatory packing
{\it Mathematics of Computation},
{\bf 27}, (122) 1973, pp. 369--377.


\bibitem
{Coxeter}   Donald Coxeter, The Problem of Apollonius,
{\it The American Mathematical Monthly},
1968, {\bf 75}, (1) 1968, pp. 5-15.






\bibitem
{Del} 	
Pierre Deligne, Notes on spinors, in: ``Quantum Fields and Strings: a Course for Mathematicians'',
 (AMS 1999) pp. 99–135.



\bibitem
{For} 	Lester R. Ford, Fractions, {\it American Mathematical Monthly}, (9) 45, 1938, pp 586--601.

\bibitem
{GLM1} 	
Ronald  L. Graham, Jeffrey C. Lagarias, Colin L. Mallows, Allan R. Wilks and Catherine H. Yan, 
Apollonian circle packings: geometry and group theory I. Apollonian group, 
{\it Discrete \& Computational Geometry} 34 (2005), 547--585 
[arXiv.org/math.MG/0010298]

%

\bibitem
{GLM4}
Ronald  L. Graham, Jeffrey C. Lagarias, Colin L. Mallows, Allan R. Wilks and Catherine H. Yan, 
Apollonian circle packings: number theory, 
{\it J. Number Theory} 100, 1--45 (2003), .  
 	[arXiv/math.NT/ 009113],  [www.math.lsa.umich.edu/~lagarias/doc/apollonian-nt.pdf]

\bibitem
{Hir} 	K.E. Hirst, The Apollonian packing of circles. 
{\it J. Lond. Math. Soc.}, 42 , 281--291 (1967).

\bibitem
{jk-d}	Jerzy Kocik, A matrix theorem on circle configuration (arXiv:0706.0372v2).

\bibitem
{jk-c}	Jerzy Kocik,  Clifford Algebras and Euclid's Parameterization of Pythagorean Triples,  
{\it Advances in Appl. Cliff. Alg.}, 16  pp. 71--93 (2007),
(available as arXiv:1201.4418 [math.NT]).


\bibitem%
{jk-k}  Jerzy Kocik,  Krawtchouk matrices from the Feynman path integral and from the split quaternions, 
{\it Contemp. Math.}, {\bf 668}, 2016, pp . 131-164.

%

\bibitem
{LMW}   Jeffrey C. Lagarias, Colin L. Mallows and Allan Wilks, 
Beyond the Descartes circle theorem, 
{\it Amer. Math. Monthly} 109, 338--361 (2002). [eprint: arXiv math.MG/0101066]

\bibitem
{Lou} Pertti Lounesto.
Clifford Algebras and Spinors (London Mathematical Society Lecture Note Series), 
Cambridge University Press,  2 ed. (2001).

\bibitem
{Man}   Benoit  Mandelbrot: The Fractal Geometry of Nature,  Freeman (1982)

\bibitem
{Mel} 	Zdzis{\l}aw A. Melzak, Infinite packings of disks. 
{\it Canad. J. Math.} 18, 838--853 (1966).

\bibitem
{Ped67}	Dan  Pedoe, On a theorem in geometry, 
{\it Amer. Math. Monthly} 74, 627--640 (1967).

\bibitem
{Ped70} 	Dan  Pedoe, Geometry, a comprehensive course, 
Cambridge Univ Press,  (1970). [Dover 1980]

\bibitem
{Por} Ian R. Porteous,
Clifford Algebras and the Classical Groups (Cambridge Studies in Advanced Mathematics),
Cambridge University Press, 1995.

\bibitem
{Por2} 
Ian R. Porteous,  Topological Geometry,
Cambridge University Press; 2nd ed. (1981).

\bibitem
{Sie}  Wac\l aw Sierpi\'nski,  Pythagorean triangles, 
{\it The Scripta Mathematica Studies}, No 9, Yeshiva Univ., New York (1962).

\bibitem
{Sod}	Fredric Soddy, The Kiss Precise. {\it Nature} 137, 1021 (1936), .

\bibitem
{TT}	Olga Taussky-Todd,  The many aspects of Pythagorean triangles,  
{\it Lin. Alg. Appl.}, 43, 285--295 (1982).








\end{thebibliography}
\end{document}